\documentclass[11pt]{amsart}

 \usepackage[margin=1in]{geometry}

\newcommand{\private}[1]{}
\usepackage{amssymb, amsmath, amscd, amsthm, color, epsfig,url, graphicx}

\usepackage[all]{xy}          
\xyoption{dvips}              

\makeatletter
\renewcommand\l@subsection{\@tocline{2}{0pt}{2pc}{5pc}{}}
\makeatother

\newcommand{\R}{{\mathbb R}}

\newcommand{\hofiber}{\operatorname{hofiber}}

\newcommand{\holim}{\operatorname{holim}}

\newcommand{\hocolim}{\operatorname{hocolim}}
\newcommand{\tfiber}{\operatorname{tfiber}}

\newcommand{\Map}{\operatorname{Map}}

\newcommand{\Emb}{\operatorname{Emb}}

\newcommand{\Imm}{\operatorname{Imm}}

\newcommand{\Conf}{\operatorname{Conf}}

\newcommand{\calX}{{\mathcal{X}}}
\newcommand{\calY}{{\mathcal{Y}}}

\newcommand{\calC}{{\mathcal{C}}}
\newcommand{\calO}{{\mathcal{O}}}
\newcommand{\calP}{{\mathcal{P}}}
\newcommand{\calQ}{{\mathcal{Q}}}

\newcommand{\calT}{{\mathcal{T}}}

\newcommand{\del}{{\partial}}

\newcommand{\Top}{\operatorname{Top}}

\theoremstyle{plain}
\newtheorem{thm}{Theorem}[section]
\newtheorem{prop}[thm]{Proposition}

\theoremstyle{definition}
\newtheorem{defin}[thm]{Definition}
\newtheorem{example}[thm]{Example}
\newtheorem{def/ex}[thm]{Definition/Example}

\theoremstyle{remark}
\newtheorem{rem}[thm]{Remark}
\newtheorem{rems}[thm]{Remarks}

\newcommand{\refS}[1]{Section~\ref{S:#1}}
\newcommand{\refT}[1]{Theorem~\ref{T:#1}}

\newcommand{\refP}[1]{Proposition~\ref{P:#1}}
\newcommand{\refD}[1]{Definition~\ref{D:#1}}

\newcommand{\refEx}[1]{Example~\ref{Ex:#1}}

\begin{document}


\title[Convergence of the Taylor tower for embeddings in $\R^n$]{A streamlined proof of the convergence of the Taylor tower for embeddings in $\R^n$}

%

\author{Franjo \v Sar\v cevi\'c}
\address{Department of Mathematics, University of Sarajevo}
\email{franjo.sarcevic@live.de}

\author{Ismar Voli\'c}
\address{Department of Mathematics, Wellesley College, Wellesley, MA}
\email{ivolic@wellesley.edu}
\urladdr{ivolic.wellesley.edu}

%


\begin{abstract}
Manifold calculus of functors has in recent years been successfully used in the study of the topology of various spaces of embeddings of one manifold in another.  Given a space of embeddings, the theory produces a Taylor tower whose purpose is to approximate this space in a suitable sense.  Central to the story are deep theorems about the convergence of this tower.  We provide an exposition of the convergence results in the special case of embeddings into $\R^n$, which has been the case of primary interest in applications.  We try to use as little machinery as possible and give several improvements and restatements of existing arguments used in the proofs of the main results. 
\end{abstract}
\maketitle
\tableofcontents

\parskip=4pt
\parindent=0cm


\section{Introduction}\label{S:Intro}


Let $M$ be a smooth manifold and let $\Emb(M,\R^n)$ be the space of smooth embeddings of $M$ in $\R^n$ (see \refD{Embedding}). For many $M$, $\Emb(M,\R^n)$ is an interesting and often a notoriously difficult space (one need not look   any further than knot theory).  One technique that has emerged to the forefront of the study of the topology of $\Emb(M,\R^n)$ in the last two decades is \emph{manifold calculus of functors}, due to Goodwillie and Weiss \cite{GW:EI2, W:EI1}.  The basic idea in this theory is to  construct the \emph{Taylor tower} for $\Emb(M,\R^n)$:
$$
\Emb(M,\R^n)\longrightarrow\big(T_\infty \Emb(M,\R^n)\to\cdots\to T_k\Emb(M,\R^n)\to\cdots \to T_0\Emb(M,\R^n)  \big).
$$
This should be thought of as an analog of the Taylor series of an analytic function;  the space $T_k\Emb(M,\R^n)$,  the $k^{th}$ Taylor stage of the tower, plays the role of the $k^{th}$ Taylor polynomial, and the space $T_\infty \Emb(M,\R^n)$, the inverse limit of the tower, plays the role of the Taylor series (see \refD{TowerStages}).  All the spaces, including $\Emb(M,\R^n)$, are meant to be regarded as functors in a certain way, but the picture presented here only takes into account the value of those functors on a single object, namely $M$ itself.

Taylor tower for embeddings has been used to great effect in a variety of situations \cite{ALTV, ALV, AT:LongPlanes2, BW:EmbConfCat, DwyerHess:LongKnots, GM:LinksEstimates, LTV:Vass, M:Emb, M:LinkNumber, M:Milnor, MV:Links, S:OKS, PAST:LinkSSCollapse, V:FTK}.
The driving force behind this success are deep results about its convergence, and this is what we will address in this paper. Convergence questions can be stated in terms of the connectivities of the various maps in the tower.  (We will say more about this, including about what we mean by ``convergence'', at the beginnings of Sections \ref{S:TowerConverge} and \ref{S:TowerConvergeToEmb}.)  The main results we will prove are as follows.

\begin{thm}[Convergence of the Taylor tower]\label{T:TowerConvergenceIntro}
Suppose $M$ is a smooth manifold without boundary of dimension $m$.  Then, for $k\geq 0$, the map
$$
T_{k+1}\Emb(M,\R^n)\longrightarrow T_{k}\Emb(M,\R^n)
$$
is
$$
\big(k(n-m-2)-m+1\big)\text{-connected.}
$$
Therefore if $n>m+2$, the connectivities increase with $k$, and hence the Taylor tower for $\Emb(M,\R^n)$ converges.
\end{thm}

Here is the other main result.
\begin{thm}[Convergence of the Taylor tower to embeddings]\label{T:TowerConvergenceToEmb2mIntro}
Suppose $M$ is a smooth closed manifold of dimension $m$.  Then, if $n>2m+2$, the Taylor tower for $\Emb(M,\R^n)$ converges to this space, i.e.~the map $\Emb(M,\R^n)\to T_{\infty}\Emb(M,\R^n)$ is a weak homotopy equivalence.
\end{thm}

Theorems \ref{T:TowerConvergenceIntro} and \ref{T:TowerConvergenceToEmb2mIntro} will immediately yield the following.

\begin{thm}
\label{T:TowerConvergenceToEmbWeakIntro}
Suppose $M$ is a smooth closed manifold of dimension $m$ and suppose $n>2m+2$.  Then, for $k\geq 0$, the map 
$$
\Emb(M,\R^n)\longrightarrow T_{k}\Emb(M,\R^n)
$$
is
$$
\text{$(k(n-m-2)-m+1)$-connected.}
$$
\end{thm}

While \refT{TowerConvergenceIntro} is the best known result about the connectivities of the maps between the stages in the Taylor tower, \refT{TowerConvergenceToEmbWeakIntro} is not the best known result about the ``rate of convergence'', namely about the connectivities between $\Emb(M,\R^n)$ and the stages.  The best such statement is due to Goodwillie and Weiss \cite{GW:EI2} who build on work of Goodwillie and Klein \cite{GK}.  The improvements over \refT{TowerConvergenceToEmbWeakIntro} are that the result is in fact true for 
\begin{itemize}
\item $n>m+2$;
\item any smooth manifold $N$ of dimension $n$ in place of $\R^n$;
\item smooth manifolds $M$ and $N$ with boundary (but embeddings must respect the boundary in a suitable sense); and
\item any interior of a compact codimension zero handlebody in a compact $M$, in which case the handle dimension (the highest dimension of a handle necessary to build the handlebody) replaces $m$.
\end{itemize}
In addition, the Goodwillie-Weiss proof does not require an a priori knowledge of the convergence of the Taylor tower. 

Our weaker version, \refT{TowerConvergenceToEmbWeakIntro}, is what we could prove without bringing in large amounts of machinery.  The proof uses cubical diagram and general position techniques, and we feel this is the minimal set of tools needed to obtain any kind of convergence result of this sort.  To improve it in most of the ways listed above requires much more background in differential topology and homotopy theory. We will say more about what is needed for such improvements in \refS{Further}.  

One improvement, however, that would be fairly easy to perform would be to upgrade the result from $\R^n$ to an arbitrary smooth manifold $N$ without boundary.  The only places we use $\R^n$ are the proofs of Propositions \ref{P:ConfFibersCubeCocartesian} and \ref{P:DerivativeConnectivity}.  They both utilize some well-known facts about projections of configuration spaces in $\R^n$ (see \refEx{Configurations}), but similar results also hold for configurations in arbitrary manifolds (see, for example, \cite[Theorem 10.3.11]{MV:Cubes}).  Nevertheless, we felt that staying grounded in $\R^n$ would be of benefit to the reader, especially in light of the fact that most of the applications of manifolds calculus of functors in recent years have been for embeddings into $\R^n$.

Even though \refT{TowerConvergenceToEmbWeakIntro} is weaker than the best known rate of convergence statement, it is better than any such weaker result that appears in the literature.  What is sometimes called the ``weak convergence'' has the connectivity of the map $\Emb(M,\R^n)\to T_{k}\Emb(M,\R^n)$ as $(k(n-2m-2)-m+1)$ (and we will in fact prove this on the way to deducing \refT{TowerConvergenceToEmb2mIntro}).  Our \refT{TowerConvergenceToEmbWeakIntro} has the best known connectivity, but the hypothesis $n>2m+2$ is not optimal (as mentioned above, $n>m+2$ is the best known one).  To see why we were able to deduce a better result than the existing weak convergence statements, see the discussion following the proof of \refT{TowerConvergenceToEmbWeak}.

We should emphasize, however, that none of our results or proofs are new in any sort of essential way.  For the most part, they are simplifications and reorganizations of existing work that can be found in \cite{G:ExcisionDiffeo, GW:EI2, M:MfldCalc, MV:Cubes}.  The novelty is mostly in the streamlining and, hopefully, in increased accessibility to these deep results in functor calculus.

Theorems \ref{T:TowerConvergenceIntro} and \ref{T:TowerConvergenceToEmbWeakIntro} suffice for many of the existing applications of manifold calculus, and we hope that having proofs of these results that are as accessible as possible might generate further applications.  
For example, the authors plan to extend these results to the space of $r$-immersions (immersions with up to $r-1$ self-intersections) of $M$ in $\R^n$, and believe that the same strategy of proof will carry over to that situation.  The investigation of the low stages of the Taylor tower for $r$-immersions was already undertaken in \cite{SSV:r-immersions}, and that paper indeed utilizes many of the techniques that appear here.


\subsection{Required background}

Throughout the paper, we try to bring in as little machinery as possible.  The main results are ultimately statements about cubical diagrams and their homotopy (co)limits, so we will review some of the most basic features of these topics in \refS{Cubes}.  However, this review is not exhaustive; we will in particular assume some familiarity with homotopy (co)limits of punctured cubes.  The definition of homotopy (co)limits for these kinds of diagrams is fortunately much simpler and more intuitive than in the case of general diagrams.  The reader should consult Sections 5.3 and 5.7 of \cite{MV:Cubes} for details, but a substantial detour into that material should not be necessary.  The culmination of the cubical diagrams story are the Blakers-Massey Theorems, recalled here as Theorems \ref{T:B-M} and \ref{T:GeneralB-M}.  We will not need anything about these results beyond their statements.

Even though the underlying theory here is calculus of functors, we will require almost no category theory (this is essentially because we will only be concerned with the value of a certain functor on a single object; see comments at the beginning of \refS{EmbeddingsTower}).  Familiary with the basic language of functors and natural transformations will help the reader, but is not a prerequisite for understanding the paper.  The only major result from manifold calculus of functors we will implicitly use is the classification of homogeneous functors, but the reader need not be familiar with its statement.  In fact, we will not even recall it here since of interest to us will be only one of its consequences, \refP{LayerAndDerivative}.  Readers interested in learning more about manifold calculus of functors should look at the foundational papers \cite{GW:EI2, W:EI1} as well as the expository work  \cite{M:MfldCalc} and \cite[Section 10.2]{MV:Cubes}.

The reader should be comfortable with (spaces of) embeddings and should have some basic dexterity with differential topology, including smooth manifolds, handle decomposition, transversality, and general position.  We will give some references for background reading in the places where we invoke these notions.


\subsection{Organization of the results}

Here is an organizational chart of the results to help the reader navigate the paper:
{\small
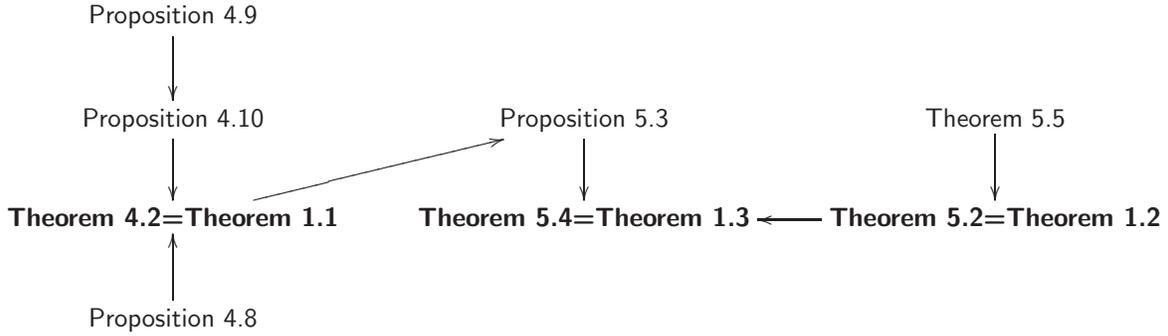
$$
\xymatrix{
 \text{\refP{ConfFibersCubeCocartesian}} \ar[d]  &  &  \\
 \text{\refP{DerivativeConnectivity}} \ar[d] & 
\text{\refP{ConnectivityAssumeConvergence}}\ar[d] & \text{\refT{TowerConvergenceToEmbCartPQ}} \ar[d]\\
  \text{\bf \refT{TowerConvergence}=\refT{TowerConvergenceIntro}}\ar[ur] & \text{\bf \refT{TowerConvergenceToEmbWeak}=\refT{TowerConvergenceToEmbWeakIntro}}& \text{\bf \refT{TowerConvergenceToEmb2m}=\refT{TowerConvergenceToEmb2mIntro}} \ar[l]\\
\text{\refP{LayerAndDerivative}}\ar[u] & &
}
$$
}

The three theorems in bold are the main results.  \refT{TowerConvergence}, which is the same as \refT{TowerConvergenceIntro} above (with a shift in index), says that the Taylor tower for embeddings converges (to something, but we do not know to what at that point).  \refS{TowerConverge} is dedicated to the proof of that statement.  This then becomes one of the principal inputs, via \refP{ConnectivityAssumeConvergence}, into \refT{TowerConvergenceToEmbWeak}, which is the same as \refT{TowerConvergenceToEmbWeakIntro} above.  As mentioned before, that result can be interpreted as providing the ``rate of convergence'' of the Taylor tower.  The other input into \refT{TowerConvergenceToEmbWeak} is \refT{TowerConvergenceToEmb2m}, a version of which appears above as \refT{TowerConvergenceToEmb2mIntro}.  The main ingredient for that result is \refT{TowerConvergenceToEmbCartPQ}, and this is in some sense the most substantial result of the paper.


\subsection{Acknowledgments}
The second author would like to thank the Simons Foundation for its support. 


\section{Cubical diagrams and the Blakers-Massey Theorem}\label{S:Cubes}


In this section, we provide the necessary background on cubical diagrams.  As mentioned before, the exposition is self-contained except for the definition of the homotopy (co)limit of a punctured diagram.  A reader familiar with general homotopy (co)limits thus has all the necessary prerequisites, and a reader who is not should consult \cite[Sections 5.3 and 5.7]{MV:Cubes} (or the foundational work on cubical diagrams \cite{CalcII}).  
Homotopy colimits of punctured cubes will be only be needed when we invoke the Blakers-Massey Theorems (Theorems \ref{T:B-M} and \ref{T:GeneralB-M}) since their hypotheses ask how \emph{cocartesian} certain cubes are (see \refD{CartesianCube}), and this is a concept that uses homotopy colimits.

We will also on occasion use the language of categories and functors, but we will do so lightly enough so that it is possible to read this section without prior experience with those topics.


We begin by recalling the standard notion of a homotopy fiber.  For spaces $A$ and $B$, we will denote by $\Map(A,B)$ the space of maps from $A$ to $B$ topologized using the compact-open topology.

\begin{defin}\label{D:Hofiber}\ 
The \emph{mapping path space} $P_f$ of a map $f\colon X\to Y$ is the subspace of $X\times \Map(I,Y)$ given by 
$$
P_f=\{(x,\alpha)\colon \alpha(0)=f(x)  \}.
$$
The \emph{homotopy fiber} $\hofiber_y f$ of a map $f\colon X\to Y$ over $y\in Y$ is the fiber of the map $P_f\to Y$ which sends $(x,\alpha)$ to $\alpha(1)$.
\end{defin}

When the basepoint $y$ is understood from the context, we will just write $\hofiber f$ instead of $\hofiber_y f$.

Here is a standard result about homotopy fibers. For details, see \cite[Section 2.2]{MV:Cubes}.

\begin{prop}\label{P:FibrationHofiber}
If $f$ is a fibration, then its fiber and homotopy fiber are homotopy equivalent.
\end{prop}

We will often write ``(homotopy) fiber'' when the map whose homotopy fiber we are considering is a fibration. 

Connectivity of spaces and maps will be central throughout.  For more on this topic, including the proof of the \refP{CompositionConnectivity} below, see \cite[Section 2.6]{MV:Cubes}.

\begin{defin}\label{D:k-connectedSpace}\ 
\begin{itemize}
\item A nonempty space $X$ is \emph{$k$-connected} if $\pi_i(X,x)=0$ for $0\leq i\leq k$ and for all choices of basepoint $x\in X$.  An infinitely connected space is said to be \emph{weakly contractible}.
\item A map $f\colon X\to Y$ is \emph{$k$-connected} if its homotopy fiber (over any point $y\in Y$) is $(k-1)$-connected.  Equivalently, if $X\neq\emptyset$, $f$ is $k$-connected if, for all $x\in X$, the induced map
$$
f_*\colon \pi_i(X,x)\longrightarrow \pi_i(Y,f(x))
$$
is an isomorphism for all $i<k$ and a surjection for $i=k$.  An infinitely connected map is a \emph{weak equivalence}.
\end{itemize}
\end{defin}

\begin{example}\label{Ex:SphereWedgeConnectivity}
The sphere $S^k$ is $(k-1)$-connected.  If $X_1$ is $k_1$-connected and $X_2$ is $k_2$-connected, then the wedge $X_1\vee X_2$ is $\min\{k_1,k_2\}$-connected.  In particular, $\vee_l S^k$ is $(k-1)$-connected for any $l\geq 1$.
\end{example}

\begin{example}\label{Ex:Configurations}
For $X$ a space, define the \emph{configuration space of $k$ points in $X$} to be
$$
\Conf(k,X)=\{(x_1, x_2, ..., x_{k})\in X^{k} \colon x_i\neq x_j \text{ for } i\neq j\}.
$$
There are projection maps
\begin{equation}\label{E:ConfProjection}
\Conf(k,X) \longrightarrow \Conf(k-1,X)
\end{equation}
given by omitting a point.  When $X$ is a manifold, these projections are fibrations whose fiber is $X-\{b_1,...,b_{k-1}\}$, where $(b_1,...,b_{k-1})$ is a basepoint in $\Conf(k-1,X)$ \cite{FN:ConfFibration}.  When $X=\R^n$, this fiber is homotopy equivalent to a wedge of spheres, i.e.~there is a fibration sequence
\begin{equation}\label{E:ConfFibration}
\bigvee_{k-1}S^{n-1}\longrightarrow \Conf(k,\R^n) \longrightarrow \Conf(k-1,\R^n).
\end{equation}
From \refEx{SphereWedgeConnectivity}, the sphere $S^{n-1}$ is $(n-2)$-connected, and this means, by  \refD{k-connectedSpace} and \refP{FibrationHofiber}, that the projection map between the configuration spaces in \eqref{E:ConfFibration} is $(n-1)$-connected.
\end{example}

The following is immediate from the definitions.

\begin{prop}\label{P:CompositionConnectivity}
For maps $f\colon X\to Y$ and $g\colon Y\to Z$, 
\begin{enumerate}
\item If $f$ and $g$ are $k$-connected, then $g\circ f$ is $k$-connected;
\item If $f$ is $(k-1)$-connected and $g\circ f$ is $k$-connected, then $g$ is $k$-connected;
\item If $g$ is $(k+1)$-connected and $g\circ f$ is $k$-connected, then $f$ is $k$-connected.
\end{enumerate}
\end{prop}

\begin{rem}\label{R:CompositionConnectivity}
Of special interest to us will be the application of \refP{CompositionConnectivity}(1) in the proof of \refP{ConnectivityAssumeConvergence} where $f$ will have higher connectivity than $g$, or will even be infinitely connected, i.e.~an equivalence, in which case $g\circ f$ has the same connectivity as $g$.
\end{rem}


Let $\underline k=\{1,2,...,k\}$ and let $\calP(\underline k)$ be the set of subsets of $\underline k$.  Let $\calP_0(\underline k)$ be the set of all nonempty subsets of $\underline k$ and let $\calP_1(\underline k)$ be the set of all proper subsets of $\underline k$.  All three sets can be regarded as categories (posets) with respect to inclusions.

Let $\Top$ be the category of topological spaces with maps as morphisms.

\newpage

\begin{defin}\label{D:Cube}\ 
\begin{itemize}
\item A \emph{$k$-cube}, or a \emph{cubical diagram of dimension $k$}, is a functor
\begin{align*}
\calX\colon \calP(\underline k) & \longrightarrow \Top 
\\
S &\longmapsto X_S
\end{align*}
which means that to each $S\in\calP(\underline k)$, $\calX$ assigns a space $X_S$ and to each inclusion $T\hookrightarrow S$, $\calX$ assigns a map
$$
X_T\longrightarrow X_S
$$ such that the square 
$$
\xymatrix{
T \ar[r]\ar[d]_-{\chi(T)} &  S \ar[d]^-{\chi(S)} \\
X_T \ar[r] &  X_S
}
$$
commutes.
\item A \emph{punctured $k$-cube}, or a \emph{punctured cubical diagram of dimension $k$}, is the same except the domain is $\calP_0(\underline k)$ or $\calP_1(\underline k)$.
\end{itemize}
\end{defin}

We will often abbreviate $X_S$ by writing the subscript without braces and commas, so, for example, $X_{\{1,2,3\}}$ will be written simply as $X_{123}$.

A cubical diagram can be represented in the shape of a cube (hence the name), so, for example, a 3-cube is a diagram
$$
\xymatrix@=10pt{
   X_\emptyset \ar[dd]\ar[dr]
    \ar[rr]
   &           &   X_1 \ar'[d][dd]           \ar[dr]  &                  \\
          &  X_2 \ar[dd] \ar[rr] &             & X_{12}   \ar[dd] \\
   X_3 \ar'[r][rr] \ar[dr] &        &   X_{13}
\ar[dr] &                   \\
          &   X_{23} \ar[rr]     &                    &   X_{123}
}
$$

Notice that there are two kinds of punctured cubes, either missing $X_\emptyset$, the ``initial'' space, or $X_{\underline k}$, the ``final'' space.  It will always be clear from the context which one we mean.

\begin{defin}\label{D:Face}
For $\calX$ a $k$-cube and $T\subset S\subset \underline k$, the \emph{face $\del^S_T\calX$} is the $(|S|-|T|)$-cube given by the assignment $U\mapsto X_{U}$ for $U\subset S-T$.
\end{defin}

Faces of $\calX$ in particular always contain $X_\emptyset$. 
The two types of faces we will encounter here are square faces and those where $T=\emptyset$.  The latter will be abbreviated by $\del^S\calX$ (instead of $\del^S_\emptyset\calX$).  

\begin{defin}\label{D:MapOfCubes}
Suppose $\calX$ and $\calY$ are $k$-cubes or punctured $k$-cubes.  Then a \emph{map} or a \emph{natural transformation} $F\colon\calX\to\calY$ of (punctured) cubes $\calX$ and $\calY$ is a collection of maps
$$
f_S\colon X_S\longrightarrow Y_S
$$
for all $S\subset \underline k$ such that, for all $S\subset T$, the square
$$
\xymatrix{
X_S \ar[r]^-{f_S}\ar[d] &  Y_S \ar[d] \\
X_T \ar[r]^-{f_T} &  Y_T
}
$$
commutes.  
\end{defin}
Another way to state the above definition is to say that $F$ produces a $(k+1)$-cube where $S\mapsto X_S$ if $k+1\notin S$ and $S\mapsto Y_{S -\{k+1\}}$ if $k+1\in S$.

Using this definition, one can also think of a $k$-cube as a map of $(k-1)$-cubes ``in $k$ directions''.  For example, the square
$$
\xymatrix{
X_\emptyset \ar[r]\ar[d] &  X_1 \ar[d] \\
X_2 \ar[r] &  X_{12}
}
$$
is a map $F$ of 1-cubes in two ways:
$$
\xymatrix{
X_\emptyset \ar[rr] & \ar[d]^-F & X_1 \\
X_2 \ar[rr] &  & X_{12}
}
\ \ \ \ \ \ \ \ \ 
\ \ \ 
\xymatrix{
X_\emptyset \ar[dd]  & X_1 \ar[dd]\\
  \ar[r]^-F  &  \\
X_2   & X_{12}
}
$$
With this in hand, we can then also consider the $(k-1)$-cube obtained by taking homotopy fibers, in some direction, of a $k$-cube $\calX$ considered as a map of two $(k-1)$-cubes.  The homotopy fibers should be taken over basepoints that are compatible in the sense that each map in the cube sends basepoint to basepoint.  The maps in the cube of homotopy fibers are induced by the maps in $\calX$; we leave it to the reader to work out some simple cases of this (or consult \cite{MV:Cubes}).

We can keep going with taking homotopy fibers until we get to a 0-cube, i.e.~a space.  This space is called the \emph{iterated (homotopy) fiber}.  Since there are various directions in which we can take the homotopy fibers at each stage, one can a priori reach different total fibers, but this is not the case; taking homotopy fibers in different directions produces homotopy equivalent total fibers (see proof of Proposition 5.5.4 in \cite{MV:Cubes}).  For example, the square diagram above yields two iterated fibers, $\hofiber \gamma$ and $\hofiber h$, where $\gamma$ and $h$ are maps induced on homotopy fibers of pairs of parallel maps in the square:
$$
\xymatrix{
\hofiber\gamma \ar[r] & \hofiber \alpha \ar[r]^-\gamma\ar[d] & \hofiber \beta\ar[d] \\
& X_\emptyset \ar[r]^-f\ar[d]_-\alpha  & X_1\ar[d]^-\beta\\
& X_2\ar[r]^-g & X_{12}
}
\ \ \ \ \ \ \ \ \ \ \ \ \ \ \ 
\xymatrix{
\hofiber h \ar[d] & & \\
\hofiber f \ar[r] \ar[d]^-h & X_\emptyset \ar[r]^-f\ar[d]_-\alpha  & X_1\ar[d]^-\beta  \\
\hofiber g \ar[r] & X_2\ar[r]^-g & X_{12}
}
$$
But then we have that $\hofiber \gamma\simeq \hofiber h$.

We will relate the iterated fiber to another important idea, the \emph{total fiber}, shortly.

Of particular importance is the \emph{homotopy limit} (or \emph{homotopy pullback}) of a punctured cube missing the initial space and the \emph{homotopy colimit} (or \emph{homotopy pushout}) of a punctured cube missing the final space.  We will denote these spaces by
$$
\underset{S\in \calP_0(\underline k)}{\holim} X_S
\ \ \ 
\text{and}
\ \ \ 
\underset{S\in \calP_1(\underline k)}{\hocolim} X_S,
$$
respectively, and will sometimes abbreviate them as $\holim\calX$ and $\hocolim\calX$ with the understanding that when we take the homotopy limit of $\calX$, then this is a punctured cube missing the initial space, and if we take the homotopy colimit of $\calX$, then this is a punctured cube missing the final space.

Homotopy limit and colimit are subspaces of the products of spaces of maps $\Delta^{|S|}\to X_S$ and of the disjoint union of $\Delta^{k-|S|}\times X_S$, respectively, and can be thought of as ``fattened up'' limit (pullback) and colimit (pushout).  This fattening affords them homotopy invariance (which limit and colimit do not have) and we will take advantage of this freely.  Namely, we will habitually replace diagrams by homotopy or weakly equivalent ones.

For detailed definitions and examples of homotopy (colimits) of punctured cubes, see Sections 5.3 and 5.7 of \cite{MV:Cubes}.

A punctured $(\geq\! 1)$-cube $\calX$ contains punctured subcubes $\calY$  of lower dimensions.  For example, the punctured 3-cube
$$
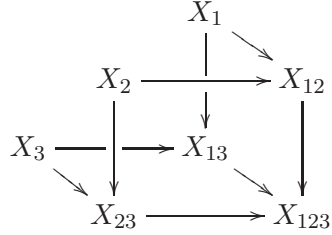
\xymatrix@=10pt{
   &           &   X_1 \ar'[d][dd]           \ar[dr]  &                  \\
          &  X_2 \ar[dd] \ar[rr] &             & X_{12}   \ar[dd] \\
   X_3 \ar'[r][rr] \ar[dr] &        &   X_{13}
\ar[dr] &                   \\
          &   X_{23} \ar[rr]     &                    &   X_{123}
}
$$
contains the punctured 2-cube $X_1\to X_{12}\leftarrow X_2$.

Then one gets induced maps
\begin{equation}\label{E:SubdiagramMap}
\holim \calX \longrightarrow \holim\calY
\end{equation}
which are in fact fibrations (see for example \cite[Theorem 8.6.1]{MV:Cubes}).  This will be useful on several occasions later.

If a space $X$ maps into a punctured cube $\calX$ (by which we mean that there are maps $X\to X_S$ for all $S$ and all the evident triangles commute), then there are induced maps
\begin{align}
X & \longrightarrow \holim\calX, \label{E:SpaceToDiagramMapHolim}\\
X & \longrightarrow \hocolim\calX. \label{E:SpaceToDiagramMapHocolim}
\end{align}

There are also canonical maps
\begin{align*}
a(\calX)\colon & X_\emptyset \longrightarrow \underset{S\in \calP_0(\underline k)}{\holim} X_S \\
b(\calX)\colon &  \underset{S\in \calP_1(\underline k)}{\hocolim} X_S \longrightarrow X_{\underline k}
\end{align*}
due to the fact that $X_\emptyset$ and $X_{\underline k}$ map to and admit maps from the rest of the cube, respectively.  A good picture to keep in mind for $a(\calX)$ is 
\begin{equation}\label{E:a(X)Diagram}
\xymatrix{
X_\emptyset\ar[rr]\ar[dd]\ar[dr]^-{a(\calX)} & & X_1\ar[dd] \\
&  \underset{S\in\calP_0(\underline 2)}{\holim} X_S\ar[ur]\ar[dl]  &   \\
X_2\ar[rr] & & X_{12}
 & 
}
\end{equation}
A similar ``dual'' diagram can be drawn for $b(\calX)$.

\begin{defin}\label{D:CartesianCube}\ 
\begin{itemize}
\item A $k$-cube is \emph{$c$-cartesian} if $a(\calX)$ is $c$-connected.  If $c=\infty$, the cube is \emph{cartesian}.  If each face of the cube of dimension $\geq 2$ is cartesian, then the cube is \emph{strongly cartesian}.
\item A $k$-cube is  \emph{$c$-cocartesian} if $b(\calX)$ is $c$-connected.  If $c=\infty$, the cube is \emph{cocartesian}.  If each face of the cube of dimension $\geq 2$ is cocartesian, then the cube is \emph{strongly cocartesian}.
\end{itemize}
\end{defin}

One reference for the proof of the following is \cite[Proposition 5.4.12]{MV:Cubes}.  Recall from earlier that a cube can be regarded as a map of cubes of a lower dimension.

\begin{prop}\label{P:FiberCartesian}
A $k$-cube is $c$-cartesian if and only if the $(k-1)$-cube of its homotopy fibers (taken in some direction and over compatible basepoints) is $c$-cartesian.
\end{prop}

It turns out that for a cube to be strongly (co)cartesian, it suffices for each of its square faces to be (co)cartesian (see remarks following Definitions 5.4.18 and 5.8.18 in \cite{MV:Cubes}).  Here is an example that will reappear as a cocartesian square face in the proof of \refP{ConfFibersCubeCocartesian}.

\begin{example}\label{Ex:Intersection-Union}
Suppose $X$ is a space and $U, V$ are open subsets.
 Then the square
$$
\xymatrix{
U\cap V \ar[r]\ar[d] & U \ar[d]\\
V \ar[r] & U\cup V
}
$$
where all the maps are inclusions, is cocartesian.  We leave it to the reader to work out the details (or see \cite[Example 3.7.5]{MV:Cubes}).  We will refer to this square as the ``intersection-union'' square.
\end{example}

For $a(\calX)$ to be $c$-connected means that its homotopy fiber is $(c-1)$-connected (\refD{k-connectedSpace}).  This homotopy fiber has a special name. 

\begin{defin}\label{D:TotalFiber}
The \emph{total fiber of a cube $\calX$}, $\tfiber(\calX)$, is 
$$
\tfiber(\calX)=\hofiber(a(\calX)).
$$
\end{defin}

As alluded to before, there is a relationship between the iterated fiber and the total fiber.  For the proof of the following, see for example \cite[Proposition 5.5.4]{MV:Cubes}.

\begin{prop}\label{P:IteratedFiber}
Total fiber and iterated fiber are homotopy equivalent.
\end{prop}

\begin{rem}
There are correponding notions of the total and iterated \emph{co}fiber, with dual results holding, but we will not need them here.  For more details, see \cite[Section 5.9]{MV:Cubes}.
\end{rem}


We can finally state the main results we will use in our proofs. The following is the Generalized Blakers-Massey Theorem due to Ellis and Steiner \cite{ES}, although it first appears in this form as \cite[Theorem 2.3]{CalcII}. It will be used in the proof of \refP{DerivativeConnectivity}.

\begin{thm}[Generalized Blakers-Massey Theorem]\label{T:B-M}
Suppose $k\geq 1$ and let
\begin{align*}
\calX\colon \calP(\underline k) & \longrightarrow \Top \\
S &\longmapsto X_S
\end{align*}
be a strongly cocartesian $k$-cubical diagram. Suppose the maps 
$$
X_\emptyset \longrightarrow X_i \ \ \ \text{are}\ \ \ \text{$k_i$-connected,}\ \ \ 1\leq i\leq k.
$$
Then $\calX$ is
$$
\big(1-k+\sum_{i\in \underline k} k_i\big)-\text{cartesian}.
$$
\end{thm}

We also have a generalization of the previous result due to Goodwillie  \cite[Theorem 2.5]{CalcII} (also see \cite[Theorem 6.2.3]{MV:Cubes}).  This result will be used in the proof of  \refT{TowerConvergenceToEmbCartPQ}.

\begin{thm}
\label{T:GeneralB-M}
Suppose $k\geq 1$ and let
\begin{align*}
\calX\colon \calP(\underline k) & \longrightarrow \Top \\
S &\longmapsto X_S
\end{align*}
be a $k$-cubical diagram. Suppose that
\begin{itemize}
\item for each nonempty subset $S$ of $\underline k$, the face $\del^S\calX$ is $k_S$-cocartesian; and
\item for $T\subset S$, $k_T\leq k_S$.
\end{itemize}
Then $\calX$ is $c$-cartesian, where $c$ is the minimum of 
$$
1-k+\sum k_{T_\alpha}
$$ 
taken over all partitions $\{ T_\alpha\}$ of $\underline k$.
\end{thm}

\begin{example} In the 3-cube 
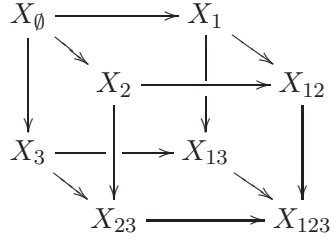
\begin{equation}
\xymatrix@=10pt{
   X_\emptyset \ar[dd]\ar[dr]
    \ar[rr]
   &           &   X_1 \ar'[d][dd]           \ar[dr]  &                  \\
          &  X_2 \ar[rr] \ar[dd]  &             & X_{12}   \ar[dd] \\
   X_3 \ar'[r][rr] \ar[dr] &        &   X_{13}
\ar[dr] &                   \\
          &   X_{23} \ar[rr]      &                    &   X_{123}
}
\end{equation}
satisfying the hypotheses of the theorem (i.e.~each face is at least as cocartesian as all the faces it contains), one looks at 
\begin{enumerate}
\item How cocartesian the entire cube is; this is $k_{123}$.

\item The sum of the connectivity of the map $X_\emptyset\to X_1$ and how cocartesian the square 
$$
\xymatrix{
X_\emptyset \ar[r]\ar[d] &  X_2\ar[d]\\
X_3 \ar[r] & X_{23}
}
$$
is; this is $k_1+k_{23}$.
\item Two more numbers as in the previous item, except the indices are permuted; these are $k_2+k_{13}$ and $k_3+k_{12}$.
\item The sum of the connectivities of the maps $X_\emptyset\to X_i$; this is $k_1+k_2+k_3$.
\end{enumerate}
Then the cube is
$$
(1-3+\min\{k_{123}, k_1+k_{23}, k_2+k_{13},  k_3+k_{12}, k_1+k_2+k_3\})\text{-cocartesian}.
$$ 
\end{example}

\begin{rem}
\refT{GeneralB-M} is indeed a generalization of \refT{B-M} since, if $\calX$ is strongly cocartesian, then all its faces are cocartesian, i.e.~infinitely cocartesian.  The minimum is thus achieved with the sum $\sum k_i$ since all the other numbers are infinity. 
\end{rem}


\section{Taylor tower for embeddings}\label{S:EmbeddingsTower}


The story of Taylor towers in manifold calculus of functors is deep and rich.  It applies to various functors from the category of open subsets of a manifold to the category of spaces or spectra that satisfy certain technical conditions.  Our goal here, however, is to elucidate the particular case of the space/functor of embeddings without bringing in all the machinery.  In addition, we are only concerned with embeddings of $M$ and not the functoriality with respect to its open subsets, so what we are doing here can be thought of as just analyzing the Taylor tower for one particular value of the embedding functor, namely on the open set that is $M$ itself.  One upshot of narrowing our focus so much is that it liberates us from requiring a lot of background from category theory or homotopy theory. 

As a result, this section will neither require nor provide much background on manifold calculus of functors.  The reader who wishes to see the bigger picture or who wants to be able to place the results in this paper into a larger context of manifold calculus should consult the expository writing \cite{M:MfldCalc} and \cite[Section 10.2]{MV:Cubes} in addition to papers \cite{GW:EI2, W:EI1} that set up the theory.  


\begin{defin}\label{D:Embedding}
Let $M$ be a smooth manifold without boundary of dimension $m$.  
\begin{itemize}
\item An \emph{embedding} of $M$ in $\R^n$ is a smooth map that is a homeomorphism onto its image and whose derivative is injective.
\item The \emph{space of embeddings} $\Emb(M,\R^n)$ is the subspace of the space of smooth maps from $M$ to $\R^n$ consisting of smooth embeddings of $M$ in $\R^n$ (the space of smooth maps is topologized using the Whitney $\calC^\infty$ topology).
\end{itemize}
\end{defin}

\begin{rem}
We will also mention the space of \emph{immersions} $\Imm(M,\R^n)$ in passing (\refEx{T_0andT_1}), i.e.~the space of smooth maps from $M$ to $\R^n$ with injective derivative.  This space is topologized the same way as the space of embeddings, and the latter can then be regarded as a subspace of the space of immersions.
\end{rem} 

We can now define the stages of the Taylor tower for $\Emb(M,\R^n)$.  We will first give the original definition, \refD{TowerStages} below, in terms of a homotopy limit of a certain infinite diagram of spaces of embeddings.  Appreciating this definition unfortunately requires dexterity with homotopy limits of general diagrams, so we will also offer an alternative description in terms of iterated punctured cubical diagrams.  The latter has the advantage that it is easier to digest, but it has the defect that it is not functorial in a certain way.  However, in line with the remarks above, this is not an issue for the purposes of this paper since we are only interested in embeddings of $M$ and not of its open subsets.  

Let $\calO(M)$ be the poset of open subsets of $M$ (with inclusions) and let $\mathcal{O}_k(M)$ be the subposet of  $\mathcal{O}(M)$ consisting of open subsets of $M$ diffeomorphic to up to $k$ disjoint balls.

\begin{defin}\label{D:TowerStages}
For $k\geq 0$, define the \emph{$k^{th}$ stage of the Taylor tower for $\Emb(M,\R^n)$}, denoted by $T_k\Emb(M,\R^n)$, to be the homotopy limit
\begin{equation}\label{E:T_kHolim}
T_k\Emb(M,\R^n)=\underset{V\in \mathcal{O}_k(M)}{\holim} \Emb(V,\R^n)
\end{equation}
\end{defin}


The stage $T_k\Emb(M,\R^n)$ mimics the $k^{th}$ Taylor polynomial of an ordinary analytic function.  It is polynomial of degree $k$ in the sense that its higher derivatives (suitably defined) vanish.  This stage should be thought of as trying to approximate embeddings of $M$ from embeddings of pieces of $M$.  See \cite[Section 4]{M:MfldCalc} for more details and helpful intuition.  


\begin{example}\label{Ex:T_0andT_1}
It is easy to see from the definition that
$$
T_0\Emb(M,\R^n)\simeq\ast,
$$
a one-point space.
It also turns out that 
$$T_1\Emb(M,\R^n)\simeq\Imm(M,\R^n).
$$  For more detail, see \cite[Example 2.3]{W:EI1} or \cite[Examples 4.8 and 4.14]{M:MfldCalc}.  This equivalence can be thought of as saying that the  ``linearization'' of the space of embeddings is the space of immersions.  
\end{example}

Now suppose $M$ is closed (so we can utilize its handle structure) and $A_i$, $1\leq i\leq k+1$, are disjoint closed subsets $A_i$ of $M$ such that removing any number of them from $M$ reduces its handle dimension (\emph{handle dimension} of $M$ is the highest index of a handle necessary to build $M$).  This can always be done (using the handle structure of $M$) and we thus obtain a punctured cubical diagram
\begin{align}
\mathcal H_k\colon \calP_0(\underline{k+1}) & \longrightarrow \Top \label{E:HolesDiagram}\\
S &\longmapsto \Emb(M-\bigcup_{i\in S}A_i, \R^n) \notag
\end{align}
where an inclusion $T\hookrightarrow S$ goes to the restriction
$$
\Emb(M-\bigcup_{j\in T}A_j, \R^n)\longrightarrow \Emb(M-\bigcup_{i\in S}A_i, \R^n).
$$
By construction, each of the manifolds being embedded has handle index smaller than that of $M$.  Applying $T_k(-)$ to this punctured cube produces a new punctured cube 
\begin{align}
\calT_k\colon \calP_0(\underline{k+1}) & \longrightarrow \Top \label{E:HolesCube}\\
S &\longmapsto T_k\Emb(M-\bigcup_{i\in S}A_i, \R^n) \notag
\end{align}
where an inclusion $T\hookrightarrow S$ now goes to the map
$$
T_k\Emb(M-\bigcup_{j\in T}A_j, \R^n)\longrightarrow T_k\Emb(M-\bigcup_{i\in S}A_i, \R^n).
$$
 ($T_k(-)$ is really a functor from spaces to spaces, so it makes sense to apply it to a cube of embedding spaces and obtain another cube of spaces.)
For each $S\subset \underline{k+1}$, there is also a canonical restriction map
$$
\Emb(M,\R^n)\longrightarrow \Emb(M-\bigcup_{i\in S}A_i, \R^n)
$$
and an induced map  
$$
T_k\Emb(M,\R^n)\longrightarrow T_k\Emb(M-\bigcup_{i\in S}A_i, \R^n),
$$
which commutes with the maps in $\calT_k$.  This means, by \eqref{E:SpaceToDiagramMapHolim}, that there is a map
\begin{equation}\label{E:T_kSame}
T_k\Emb(M,\R^n)\longrightarrow \underset{S\in \calP_0(\underline{k+1})}{\holim} T_k\Emb(M-\bigcup_{i\in S}A_i, \R^n).
\end{equation}
It turns out that, because of the properties of $T_k$ and polynomial functors, this map is an equivalence.  For details, see \cite[Example 10.2.18]{MV:Cubes}.

The process can now be repeated so that for each $S\subset \underline{k+1}$, $k+1$ disjoint closed subsets of 
$M-\bigcup_{i\in S}A_i$ are chosen so that removing any number of those reduces the handle dimension of $M-\bigcup_{i\in S}A_i$.  Thus each  $T_k\Emb(M-\bigcup_{i\in S}A_i, \R^n)$ in $\calT_k$ can now be written as a homotopy limit of a punctured cubical diagram of embeddings of manifolds whose handle index is smaller than that of $M-\bigcup_{i\in S}A_i$.

Continuing to 	``punch holes'' in $M$ like this, at each stage we have new smooth submanifolds $L$ of $M$ of lower handle dimension than the submanifolds in the previous stage, and in each of those submanifolds we punch new $k+1$ holes that reduce its handle dimension.  We can this way reduce to only having to consider embeddings of index zero submanifolds of $M$, namely finite unions of open balls.  At that point, we let the holes be the balls themselves.  (At each stage of the process, the result is the interior of a smooth compact manifold of codimension zero, so it will have finitely many handles and the process will thus terminate in finitely many steps.) 

Furthermore, all we need is to consider up to $k$ balls since what a polynomial functor such as $T_k$ does on more than $k$ balls is determined by what it does on up to $k$ balls (for details, see for example the proof of Theorem 10.2.14 in \cite{MV:Cubes}).
Finally, again because of properties of polynomial functors, there is an equivalence \cite[Theorem 6.1]{W:EI1}
\begin{equation}\label{E:SameOnBalls}
\Emb(\leq k \text{ balls}, \R^n)\stackrel{\simeq}{\longrightarrow} T_k\Emb(\leq k \text{ balls}, \R^n).
\end{equation}

\begin{example}
If $M=S^1$, then $M-\bigcup_{i\in S}A_i$ is the union of $|S|$ open 1-balls or open arcs.  Since these are already handles of index zero, we have that $T_k\Emb(S^1,\R^n)$ is the homotopy limit of the punctured cube of embeddings of arcs in $\R^n$.
\end{example}

\begin{example}
Consider $T_2\Emb(S^2,\R^n)$.  We can remove subsets $A_1$, $A_2$, and $A_3$ so that the complement of one of them is the disk $D^2$, the complement of two of them is the annulus $A$, and the complement of all three of them is the pair of pants $P$.  Then $T_2\Emb(S^2,\R^n)$ is the homotopy limit of seven spaces $T_2\Emb(L,\R^n)$ where $L$ is $D^2$, $A$, or $P$.

We can punch further holes in $A$ to that the complement of one, two, or three holes is one, two, or three disks.  Thus $T_2\Emb(A,\R^n)$ is the homotopy limit of seven spaces $\Emb(L,\R^n)$ where $L$ is now one, two, or three disks.  Similarly we can resolve $P$ so that $T_2\Emb(P,\R^n)$ can ultimately be expressed in terms of embeddings of one, two, or three disks.  

The pattern continues in a straightforward way for $T_k\Emb(S^2,\R^n)$.
\end{example}

To construct $T_{k-1}\Emb(M,\R^n)$, choices of where to punch holes are made, but these choices can be made consistently so that the next stage is constructed by using the same holes in addition to a new one at each stage of the iteration.  Thus the punctured cubical diagrams like $\calT_{k-1}$ whose homotopy limits iteratively define $T_{k-1}\Emb(M,\R^n)$ are subdiagrams of diagrams like $\calT_{k}$, those that define $T_{k}\Emb(M,\R^n)$.  It then follows (see \eqref{E:SubdiagramMap}) that there is an induced fibration, for $k\geq 1$,
\begin{equation}\label{E:T_k->T_{k-1}}
T_{k}\Emb(M,\R^n)\longrightarrow T_{k-1}\Emb(M,\R^n)
\end{equation}
(Another way to think about this is in terms of \refD{TowerStages}; the above map is induced by the inclusion of posets $\mathcal{O}_{k-1}(M)\hookrightarrow \mathcal{O}_k(M)$.)  The first big question we will try to answer in \refS{TowerConverge} is that of the connectivity of this map.

Another observation is that there is a canonical map
\begin{equation}\label{E:Emb->T_k}
\Emb(M,\R^n)\longrightarrow T_{k}\Emb(M,\R^n)
\end{equation}
induced by the restriction maps
$$
\Emb(M,\R^n)\longrightarrow \Emb(L, \R^n)
$$
where $L$ is any of the submanifolds of $M$ obtained by punching holes in $M$ in the manner described above.  (Again, another way to think about this is in terms of \refD{TowerStages}; the above map is induced by the restriction of an embedding of $M$ to an embedding of $V\in \mathcal{O}_k(M)$.)  The second big question we will try to answer in \refS{TowerConvergeToEmb} is that of the connectivity of this map.

%
%
%

Combining the maps in \eqref{E:T_k->T_{k-1}} and \eqref{E:Emb->T_k}, we get 
the \emph{Taylor tower for $\Emb(M,\R^n)$}:
\begin{equation}\label{E:TowerWithEmb}
\xymatrix{
 &  &  T_{\infty}\Emb(M,\R^n)\ar[d] \\
       &  &   \vdots \ar[d]  \\
 \Emb(M,\R^n)\ar[uurr]\ar[rr]\ar[drr]\ar[dddrr]\ar@/_3pc/[ddddrr]   &  &    T_k\Emb(M,\R^n)\ar[d] \\
       &  &    T_{k-1}\Emb(M,\R^n) \ar[d] \\
            &  &   \vdots \ar[d]  \\
            &  &    T_{1}\Emb(M,\R^n)\simeq \Imm(M,\R^n)\ar[d] \\
                &  &    T_{0}\Emb(M,\R^n)\simeq \ast
}
\end{equation}
As noted above, the maps in the tower are fibrations, and $T_\infty \Emb(M,\R^n)$ is the inverse (homotopy) limit of the tower:
$$
T_\infty \Emb(M,\R^n)=\underset{k}{\holim}\, T_k \Emb(M,\R^n)=\lim_k T_k \Emb(M,\R^n)
$$

Just as with the ordinary Taylor series, one can ask two natural questions about the Taylor tower \eqref{E:TowerWithEmb}:
\begin{enumerate}
\item Does it converge?
\item If it converges, what does it converge to?
\end{enumerate}

These questions are addressed in the next two sections.


\section{Convergence of the Taylor tower}\label{S:TowerConverge}


In this section, we tackle the convergence of the Taylor tower \eqref{E:TowerWithEmb}.
We should first be clear about what we mean by this.  A reasonable notion of convergence is one where the stages  $T_k\Emb(M,\R^n)$ approximate the inverse limit of the tower, $T_{\infty}\Emb(M,\R^n)$, better and better as $k$ increases.  In other words, the homotopy type of $T_k\Emb(M,\R^n)$ should stabilize, and a precise way to detect this would be by looking at the connectivity of the maps  $T_k\Emb(M,\R^n)\to T_{k-1}\Emb(M,\R^n)$.  We thus make the following definition.

\begin{defin}\label{D:Convergence}
The Taylor tower \eqref{E:TowerWithEmb} \emph{converges} if the connectivity of the maps 
$$
T_k\Emb(M,\R^n)\longrightarrow T_{k-1}\Emb(M,\R^n)
$$
increases with $k$.
\end{defin}

The goal of this section is to prove that this tower indeed converges under certain dimensional assumptions.  Namely, we will prove the following result.

\begin{thm}\label{T:TowerConvergence}
Suppose $M$ is a smooth manifold without boundary of dimension $m$.  Then, for $k\geq 1$, the map
$$
T_k\Emb(M,\R^n)\longrightarrow T_{k-1}\Emb(M,\R^n)
$$
is
$$
\big((k-1)(n-m-2)-m+1\big)\text{-connected}
$$
Therefore if $n>m+2$, the connectivities increase with $k$ and the Taylor tower for $\Emb(M,\R^n)$ converges.
\end{thm}

This is a restatement of \refT{TowerConvergenceIntro}, with a shift in the index.

\begin{rems}\label{R:KnotsConvergence}\ 
\begin{enumerate}
\item The condition $n>m+2$ can be thought of as the radius of convergence of the Taylor tower in a suitable sense.  For more detail, see \cite[Section 6]{M:MfldCalc} or \cite[Section 2]{GW:EI2}.  
\item 
It is not known whether the Taylor tower converges in the edge case $n=m+2$.  This case is of interest because, when $M=S^1$, the space $\Emb(S^1,\R^3)$ is precisely the space of classical knots.
\end{enumerate}
\end{rems}

To establish the connectivity of the map (fibration) in \refT{TowerConvergence}, by \refD{k-connectedSpace} we want to look at its (homotopy) fiber:

\begin{defin}\label{D:Layer}
The \emph{$k^{th}$ layer} of the Taylor tower for embeddings is the space
$$
L_k\Emb(M,\R^n)=\hofiber (T_k\Emb(M,\R^n) \longrightarrow T_{k-1} \Emb(M,\R^n)).
$$
\end{defin}

Thus if we can show that $L_k\Emb(M,\R^n)$ is $((k-1)(n-m-2)-m)$-connected, \refT{TowerConvergence} will follow.  The connectivity of this layer is closely intertwined with the connectivity of another object that plays the role of the $k^{th}$ derivative of $\Emb(M,\R^n)$.
Namely, let $\calC_k$ be the $k$-cube\footnote{The subscript $k$ now denotes a $k$-cube whereas, in $\calT_k$, it denoted a $(k+1)$-cube.  This is an unfortunate consequence of the fact that the $k^{th}$ stage of the Taylor tower is defined using a $(k+1)$-cube of embedding spaces but the $k^{th}$ layer involves a $k$-cube of configuration spaces.} given by 
\begin{align}
\calC_k \colon \calP(\underline{k}) & \longrightarrow \Top \label{E:C_kCube}\\
S   & \longmapsto \Conf(k-|S|, \R^n) \notag
\end{align}
where an inclusion $T\hookrightarrow S$ goes to a map that projects away from those configuration points indexed by $S-T$.  Thus, for example, in the square $\calC_2$, which is
$$
\xymatrix{
X_\emptyset =\Conf(2, \R^n)\ar[r]\ar[d] & X_1 =\Conf(1, \R^n)\ar[d] \\
X_2 =\Conf(1, \R^n)\ar[r] & X_{12} =\Conf(0, \R^n)\simeq *
}
$$
the maps are given as
$$
\xymatrix{
(x_1,x_2)\ar[r]\ar[d] & x_2\ar[d] \\
x_1\ar[r] & 0
}
$$
where 0 denotes the origin in $\R^n$.

\begin{defin}\label{D:Derivative}
%
Define the \emph{$k^{th}$ derivative of\,  $\Emb(M,\R^n)$ at the empty set}, denoted by $\Emb(M,\R^n)^{(k)}(\emptyset)$, to be the total fiber of $\calC_k$.
\end{defin}
\begin{rem}
The usual definition of the $k^{th}$ derivative would require us to take a cubical diagram of embeddings of balls rather than points (configuration spaces are embeddings of points), which is then equivalent to a cube of \emph{framed} configuration spaces.  However, as detailed in \cite[Theorem 10.3.3]{MV:Cubes}, we can consider just the ordinary configuration spaces since the framing part does not play a role in how connected the total fiber of the cube is.  
\end{rem}
\begin{rem}
The space $\Emb(M,\R^n)^{(k)}(\emptyset)$ plays the role of the $k^{th}$ derivative evaluated at 0 in the ordinary Taylor series.  See \cite[Section 3]{M:MfldCalc} for more intuition behind this.
\end{rem}

Here is the result that will be important in proving \refT{TowerConvergence}.

\begin{prop}\label{P:LayerAndDerivative}
Suppose  the derivative $\Emb(M,\R^n)^{(k)}(\emptyset)$ is $c_k$-connected.  Then $L_k\Emb(M,\R^n)$ is $(c_k-km)$-connected.
\end{prop}

This proposition is a consequence of a deep result, due to Weiss \cite[Theorem 8.5]{W:EI1}, about the classification of \emph{homogeneous} functors in the theory of manifold calculus of functors.  The statement and the proof of that theorem are beyond the scope of this paper.  In addition to the original source, \cite{W:EI1}, more detail can be found in \cite[Section 5]{M:MfldCalc} and \cite[Section 10.2.3]{MV:Cubes}.  For an idea about why \refP{LayerAndDerivative} holds, see the discussion following Proposition 6.1 in \cite{M:MfldCalc}.

Therefore to establish \refT{TowerConvergence}, it suffices to  establish a certain connectivity of the derivative $\Emb(M,\R^n)^{(k)}(\emptyset)$.  This derivative is the total fiber of the cube $\calC_k$, but this total fiber is by definition the homotopy fiber of the map $a(\calC_k)$ (\refD{TotalFiber}).  The question of its connectivity is thus precisely the question of how cartesian the cube $\calC_k$ is (\refD{CartesianCube}).  

To get at this, we will use \refT{B-M}, the Blakers-Massey Theorem, and not on the cube $\calC_k$ itself, but on the cube of its fibers.  So recall from the discussion following \refD{MapOfCubes} that a $k$-cube can be regarded as a map of $(k-1)$-cubes and that one can consider the $(k-1)$ cube of homotopy fibers with respect to this map.

The first order of business now is to establish the hypotheses of \refT{B-M}, namely to show that the cube of fibers is strongly cocartesian.  Recall from a remark following \refP{FiberCartesian} that a cube is strongly cocartesian if each of its square faces is cocartesian.

\begin{prop}\label{P:ConfFibersCubeCocartesian}
For $k\geq 3$, the $(k-1)$-cube $\calC_k^{\text{fib}}$ obtained by taking the (homotopy) fibers of the cube $\calC_k$ in some direction is strongly cocartesian. 
\end{prop}


\begin{proof}
Choosing a basepoint $(b_1,...,b_{k})$ in $\Conf(k,\R^n)$ bases all the configuration spaces in $\calC_k$ in a compatible way, by taking the images of this point throughout the cube (this simply means forgetting some of the components $b_i$).  Now consider the $(k-1)$-cube $\calC_k^{\text{fib}}$ obtained by taking the homotopy fibers in some direction.  Each of the squares in $\calC_k^{\text{fib}}$ thus arises by taking fibers of a 3-cubical subdiagram of $\calC_k$ that looks like, for $j\geq 3$,


{\small
\begin{equation}\label{E:ConfigurationCube}
\xymatrix@=10pt{
   \Conf(j,\R^{n}) \ar[dd]\ar[dr]
    \ar[rr]
   &           &   \Conf(j-1,\R^{n}) \ar'[d][dd]           \ar[dr]  &                  \\
          &  \Conf(j-1,\R^{n}) \ar[rr] \ar[dd]  &             & \Conf(j-2,\R^{n})   \ar[dd] \\
   \Conf(j-1,\R^{n}) \ar'[r][rr] \ar[dr] &        &   \Conf(j-2,\R^{n})
\ar[dr] &                   \\
          &   \Conf(j-2,\R^{n}) \ar[rr]      &                    &   \Conf(j-3,\R^{n})
}
\end{equation} 
}

Suppose, for simplicity of labeling, that the projection maps in this cube are given by forgetting one of the first three points, e.g.~the maps are 
{\small
\begin{equation}\label{E:ConfigurationCubePoints}
\xymatrix@=10pt{
 (x_1,x_2, x_3..., x_j)\ar@{|->}[dd]\ar@{|->}[dr]
    \ar@{|->}[rr]
   &           &   (x_2, x_3, x_4,..., x_j) \ar@{|->}'[d][dd]           \ar@{|->}[dr]  &                  \\
          &  (x_1, x_3, x_4,..., x_j) \ar@{|->}[rr] \ar@{|->}[dd]  &             & (x_3, x_4,..., x_j)   \ar@{|->}[dd] \\
   (x_1, x_2, x_4,..., x_j)  \ar@{|->}'[r][rr] \ar@{|->}[dr] &        &   (x_2, x_4,..., x_j)
\ar@{|->}[dr] &                   \\
          &   (x_1, x_4,..., x_j) \ar@{|->}[rr]      &                    &   (x_4,..., x_j)
}
\end{equation} 
}

Now take, for example, the fibers (this is the same as homotopy fibers by \refP{FibrationHofiber} since the projection maps are fibrations; see \eqref{E:ConfProjection}) in the cube \eqref{E:ConfigurationCube} vertically.  With the basepoints as discussed above, and using the discussion following \eqref{E:ConfProjection}, the square that is obtained this way is thus
\begin{equation}\label{E:FibersSquare}
\xymatrix
{
\R^n - \{b_1, b_2, b_4,..., b_j  \} \ar[r]\ar[d] & \R^n - \{b_2, b_4,..., b_j  \}\ar[d]\\
\R^n - \{b_1, b_4,..., b_j  \} \ar[r] & \R^n - \{b_4,..., b_j  \} 
}
\end{equation}
But it is evident that
\begin{align*}
\R^n - \{b_1, b_2, b_4,..., b_j  \} & = \big(\R^n - \{b_2, b_4,..., b_j  \}\big)\  \bigcap\  \big(\R^n - \{b_1, b_4,..., b_j  \}\big); \\
\R^n - \{b_4,..., b_j  \} & = \big(\R^n - \{b_2, b_4,..., b_j  \}\big)\  \bigcup\  \big(\R^n - \{b_1, b_4,..., b_j  \}\big).
\end{align*}
The square \eqref{E:FibersSquare} is thus of the ``intersection-union'' form encountered in \refEx{Intersection-Union}  and is cocartesian.

The argument clearly does not change if the fibers of the cube \eqref{E:ConfigurationCube} are taken in a different direction rather than vertically.
\end{proof}

\begin{prop}\label{P:DerivativeConnectivity}
For $k\geq 1$, the derivative $\Emb^{(k)}(\emptyset,\R^n)$ is $(k-1)(n-2)$-connected.
\end{prop}

\begin{proof}
The derivative $\Emb^{(k)}(\emptyset,\R^n)$ is the total fiber of the $k$-cube $\calC_k$.  We first treat the cases $k=1$ and $k=2$.  

For $k=1$, we want the derivative to be 0-connected.  We have
$$
\Emb^{(1)}(\emptyset,\R^n)=\hofiber(\Conf(1,\R^n)\to \Conf(0,\R^n)).
$$
Since both $\Conf(1,\R^n)$ and $\Conf(0,\R^n)$ are contractible, the homotopy fiber is contractible and in particular $0$-connected.

For $k=2$, we have
$$
\Emb^{(2)}(\emptyset,\R^n)=\tfiber
\left(
\begin{aligned}
\xymatrix{
\Conf(2,\R^n) \ar[r]\ar[d] & \Conf(1,\R^n) \ar[d] \\
\Conf(1,\R^n) \ar[r]       &  \Conf(0,\R^n)
}
\end{aligned}
\right)
$$
Taking fibers, say vertically, gives the 1-cube 
$$
S^{n-1}\longrightarrow \ast.
$$
Here we have used \eqref{E:ConfFibration} for the left vertical map and again the fact that both $\Conf(1,\R^n)$ and $\Conf(0,\R^n)$ are contractible for the right vertical map.

Finally, the fiber of this map is $S^{n-1}$, and this is therefore $\Emb^{(2)}(\emptyset,\R^n)$.  Using that $S^{n-1}$ is $(n-2)$-connected, we get precisely the desired connectivity of $\Emb^{(2)}(\emptyset,\R^n)$, namely $(2-1)(n-2)$.

For $k\geq 3$, we use  \refP{ConfFibersCubeCocartesian}.  As established there, the $(k-1)$-cube of fibers $\calC_k^{\text{fib}}$ is strongly cocartesian.  That cube is homotopy equivalent to a cube of wedges of spheres since, from \eqref{E:ConfFibration}, we have fibration sequences
$$
\bigvee_{j-1}S^{n-1}\longrightarrow\Conf(j,\R^n)\longrightarrow \Conf(j-1,\R^n).
$$
(In \refP{ConfFibersCubeCocartesian}, we did not use that the fibers were equivalent to wedges of spheres since that would obsure the fact that each square in $\calC_k^{\text{fib}}$ was of the ``intersection-union'' form.)
The maps in $\calC_k^{\text{fib}}$, namely 
$$
\bigvee_{j-1}S^{n-1}\longrightarrow \bigvee_{j-2}S^{n-1},
$$
are identity on the common spheres and send the extra one to the basepoint, and so their (homotopy) fibers are $S^{n-1}$ for each $j\geq 2$.   Hence the (homotopy) fiber is $(n-2)$-connected and the map itself is $(n-1)$-connected.
(Or one could just say that, since a wedge of $(n-1)$-spheres is $(n-2)$-connected -- see \refEx{SphereWedgeConnectivity} -- it easily follows that the above map is $(n-1)$-connected.)

In particular, the initial maps in the cube $\calC_k^{\text{fib}}$,
$$
\bigvee_{k-1}S^{n-1}\longrightarrow \bigvee_{k-2}S^{n-1},
$$
are $(n-1)$-connected.  There are $k-1$ such maps since the wedge in the target is indexed by $\underline{k-1} -\{i\}$, $1\leq i\leq k-1$.  But these are precisely the maps $X_\emptyset\to X_i$ from \refT{B-M}, the Blakers-Massey Theorem (where the cube is now $(k-1)$-dimensional, rather than $k$-dimensional).  That result then says that the $(k-1)$-cube $\calC_k^{\text{fib}}$ is 
$$
\big(1-(k-1)+\sum_{i\in \underline{k-1}} (n-1)\big)\text{-cartesian},
$$
or, rewriting and simplifying,
$$
(1+(k-1)(n-2))\text{-cartesian}.
$$
By \refP{FiberCartesian}, it then follows that the original cube $\calC_k$ is also 
$(1+(k-1)(n-2))$-cartesian. But the total fiber is one less connected than the cube is cartesian, so the total fiber of $\calC_k$, namely $\Emb^{(k)}(\emptyset,\R^n)$, is 
$$
(k-1)(n-2)\text{-connected}
$$
as desired.
\end{proof}

We finally have all the necessary pieces for the proof of the main result of this section.

\begin{proof}[Proof of \refT{TowerConvergence}]
By \refP{LayerAndDerivative}, the connectivity of $L_k\Emb(M,\R^n)$ is $c_k-km$, where $c_k$ is the connectivity of the derivative $\Emb^{(k)}(\emptyset,\R^n)$.  But the latter was in \refP{DerivativeConnectivity} found to be $(k-1)(n-2)$, which means that 
$$
L_k\Emb(M,\R^n)\ \ \ \text{is}\ \ \ \big((k-1)(n-2)-km\big)\text{-connected}.
$$
Since $L_k\Emb(M,\R^n)$ is the (homotopy) fiber of $T_k\Emb(M,\R^n)\to T_{k-1}\Emb(M,\R^n)$, this in turn means that this map has connectivity that is one higher, namely
$$
T_k\Emb(M,\R^n)\longrightarrow T_{k-1}\Emb(M,\R^n)\ \ \ \text{is}\ \ \ \big((k-1)(n-2)-km+1\big)\text{-connected}.
$$
The last number can be rewritten as $(k-1)(n-m-2)-m+1$, which is what appears in the statement of the theorem.  
\end{proof}


\section{Convergence of the Taylor tower to embeddings}\label{S:TowerConvergeToEmb}



\subsection{Statement of the convergence result 
}\label{S:MainResult}


In this section, we turn to the question of what the Taylor tower for $\Emb(M,\R^n)$ converges \emph{to}.  As in the previous section, we should be clear about what this means.
\begin{defin}\label{D:ConvergenceToEmb}
The Taylor tower for embeddings \eqref{E:TowerWithEmb} \emph{converges to $X$} if there is a weak homotopy equivalence 
$$
X\stackrel{\sim}{\longrightarrow}T_\infty\Emb(M,\R^n).
$$
\end{defin}

It turns out that, as one would hope, $X=\Emb(M,\R^n)$, under certain dimensional conditions.    First observe that one way to show convergence to a space would be to show that the connectivities of the maps from the space to $T_k\Emb(M,\R^n)$ increase with $k$.  Indeed, we have the following result.


\begin{thm}[Convergence of the Taylor tower to embeddings]\label{T:TowerConvergenceToEmb2m}
Suppose $M$ is a smooth manifold without boundary of dimension $m$.  Then, for $k\geq 1$, the map 
$$
\Emb(M,\R^n)\longrightarrow T_{k}\Emb(M,\R^n)
$$
is
$$
\text{$(k(n-2m-2)+1)$-connected.}
$$
Thus if $n>2m+2$, the Taylor tower for $\Emb(M,\R^n)$ converges to this space, i.e.~the map 
$$\Emb(M,\R^n)\longrightarrow T_{\infty}\Emb(M,\R^n)$$ is an equivalence.
\end{thm}

As mentioned in the Introduction, a stronger version of this theorem  was proved by Goodwillie and Weiss \cite[Corollary 2.5]{GW:EI2}.
The stronger result is stated as \refT{TowerConvergenceToEmb} and we will make some comments about it in \refS{Further}.

However, we can quickly improve \refT{TowerConvergenceToEmb2m} because \refT{TowerConvergence} in some sense predicts what the connectivites of the maps from $\Emb(M,\R^n)$ to the tower should be.  The better statement will follow immediately from the following.
\begin{prop}\label{P:ConnectivityAssumeConvergence}
Suppose the Taylor tower converges to $\Emb(M,\R^n)$ for some smooth manifold $M$ of dimension $m$.  Then, for $k\geq 0$, the map 
$$
\Emb(M,\R^n) \longrightarrow T_k \Emb(M,\R^n)
$$
is $(k(n-m-2)-m+1)$-connected.
\end{prop}

\begin{proof}
If the Taylor tower converges to $\Emb(M,\R^n)$, we then have a diagram 
$$
\xymatrix{
\Emb(M,\R^n)\ar[r]^-\sim \ar[dr] & T_\infty \Emb(M,\R^n) \ar[d]^-{k(n-m-2)-m+1}\\
& T_k \Emb(M,\R^n)
}
$$
The reason the vertical map has the indicated connectivity is that it is a composition of maps 
$$T_\infty\to\cdots\to T_{k+1}\Emb(M,\R^n)\to T_k \Emb(M,\R^n)$$ and the least connectivity of those maps is $k(n-m-2)-m+1$ by \refT{TowerConvergence}.  It then follows, by \refP{CompositionConnectivity}(1) (and Remark \ref{R:CompositionConnectivity} in particular), that $T_\infty \to T_k \Emb(M,\R^n)$ has that connectivity.  But now, using the same statement one last time, it follows that the composition 
$$
\Emb(M,\R^n) \longrightarrow T_k \Emb(M,\R^n)
$$
is also $(k(n-m-2)-m+1)$-connected.
\end{proof}

The improvement of \refT{TowerConvergenceToEmb2m} that brings us closer to \refT{TowerConvergenceToEmb} now follows.  This already appeared in the Introduction as \refT{TowerConvergenceToEmbWeakIntro}.

\begin{thm}\label{T:TowerConvergenceToEmbWeak}
Suppose $M$ is a smooth manifold without boundary of dimension $m$ and suppose $n>2m+2$. Then, for $k\geq 0$, the map 
$$
\Emb(M,\R^n)\longrightarrow T_{k}\Emb(M,\R^n)
$$
is
$$
\text{$(k(n-m-2)-m+1)$-connected.}
$$
\end{thm}

\begin{proof}
This follows immediately from \refP{ConnectivityAssumeConvergence} and \refT{TowerConvergenceToEmb2m}.
\end{proof}


What remains to do, therefore, is to provide the proof of \refT{TowerConvergenceToEmb2m}.  The next section is dedicated to this.


\subsection{Proof of the convergence result}\label{S:MainProof}

%
%
%
Recall from \refS{EmbeddingsTower} that the building blocks for $T_k \Emb(M,\R^n)$ are punctured cubical diagrams like $\mathcal H_k$ from \eqref{E:HolesDiagram}.  In general, at each stage of the iterative process, we have a smooth manifold $L$ and closed subsets $A_i$, $1\leq i\leq k+1$, such that cutting out any number of them from $L$ reduces the handle dimension of $L$.  The main ingredient in the proof of \refT{TowerConvergenceToEmb2m} is a statement about how connected the map from $\Emb(L,\R^n)$ to this punctured cube of embeddings of spaces $L-\cup_{i\in S}A_i$.  This is \refT{TowerConvergenceToEmbCartPQ} below, and it is an example of \emph{disjunction result} that tries to extract information about embeddings of manifolds from information about embeddings of some of its pieces.

%
%
%
%

The proof of \refT{TowerConvergenceToEmbCartPQ} will rely on certain general position and transversality arguments (a reader unfamiliar with these notions might want to look at \cite[Appendix A.2]{MV:Cubes} for a quick overview), and for this reason we will want to work with submanifolds of $L$.  So suppose the dimension of $L$ is $l$ and choose the closed sets $A_i$, $1\leq i\leq k+1$, so that they are closed smooth submanifolds of $L$ of dimension $l$ and removing them as usual reduces the handle dimension (this can all be done, as usual using the handle structure).  To keep this choice in mind, we will declare
$$
Q_i=\text{relabeling of the closed smooth submanifold $A_i$ of $L$.}
$$
We will also want to apply the Isotopy Extension Theorem to projections of embeddings, but for this we will want to  work with embeddings of disjoint submanifolds, which $L-\cup_{i\in S}Q_i$ are not. So let 
\begin{equation}\label{E:P}
P=L-\bigcup_{i\in \underline{k+1}}Q_i.
\end{equation}
Then $P$ is a smooth open codimension zero submanifold of $L$, as is each of 
$$P\cup \bigcup_{i\notin S} Q_i=L-\bigcup_{i\in S} Q_i.
$$
We also have, for an inclusion $S\hookrightarrow T$, an inclusion  
$$
P\cup \bigcup_{i\notin T} Q_i \longrightarrow P\cup \bigcup_{i\notin S} Q_i.
$$

\begin{thm}
\label{T:TowerConvergenceToEmbCartPQ}
With $Q_i$ smooth closed pairwise disjoint submanifolds of $L$ of dimension $l$ and with $P$ as described above, the $(k+1)$-cube 
\begin{align}
\calQ_k\colon \calP(\underline{k+1}) & \longrightarrow \Top \label{E:QCubeForEmb->T_kCart}\\
S &\longmapsto \Emb(P\cup \bigcup_{i\notin S} Q_i, \R^n) \notag
\end{align}
is 
$$
\text{$(k(n-2l-2)+1)$-cartesian.}
$$
\end{thm}

\begin{rem} By \refD{CartesianCube}, a way to restate \refT{TowerConvergenceToEmbCartPQ} is to say that the map 
$$
\Emb(P\cup \bigcup_{i\in \underline{k+1}} Q_i, \R^n) \longrightarrow \underset{S\in \calP_0(\underline{k+1})}{\holim} \Emb(P\cup \bigcup_{i\notin S} Q_i, \R^n)
$$
is $(k(n-2l-2)+1)$-connected since $\Emb(P,\R^n)$ is the initial space in the cube $\calQ_k$.
\end{rem}

\begin{proof}[Proof of \refT{TowerConvergenceToEmbCartPQ}]
Assuming $\Emb(P\cup \bigcup_{i\in \underline{k+1}} Q_i, \R^n)$, the initial space in $\calQ_k$, is nonempty (the situation of interest is ultimately $n>2m+2$, and an embedding will always exist in this case), we first choose a basepoint $e$ in this space.  Taking the images of $e$ throughout the cube produces compatible basepoints in all the spaces in the cube.

Now consider the map
$$
\Emb(P\cup Q_{k+1} \cup \bigcup_{i\notin S\subset \underline k} Q_i, \R^n)
\longrightarrow
\Emb(\bigcup_{i\notin S\subset \underline k} Q_i, \R^n).
$$
This is a restriction of an embedding of an open submanifold of $M$ to a closed submanifold.  By the Isotopy Extension Theorem, this map is a fibration whose (homotopy) fiber over $e$ is 
$$
\Emb\Big(P\cup Q_{k+1}, \R^n -  e(\bigcup_{i\notin S\subset \underline k} Q_i)\Big).
$$
To simplify notation for the rest of the proof, let 
\begin{align*}
Q & = P\cup Q_{k+1}, \\
Q_S & =\R^n -  e(\bigcup_{i\notin S\subset \underline k} Q_i).
\end{align*}
We can now consider the $k$-cube obtained from $\calQ_k$ by taking these homotopy fibers in the direction of the maps that project away from  $Q_{k+1}$.  Namely, we can consider the $k$-cube
\begin{align*}
\calQ_k^{\text{fib}}\colon \calP(\underline{k}) & \longrightarrow \Top \\
S &\longmapsto \Emb(Q, Q_S). \notag
\end{align*}
By \refP{FiberCartesian}, if we can show that 
 $$
 \calQ_k^{\text{fib}}  \ \ \ \text{is}\ \ \ \big((k(n-2m-2)+1)\big)\text{-cartesian},
 $$
it would follow that the cube $\calQ_k$ is also $((k(n-2m-2)+1))$-cartesian, and this is what we want to prove.

To do this, we will use \refT{GeneralB-M}.  Recalling that statement, as well as the notion of a face of a cube from \refD{Face}, we wish to know how cocartesian each face
$
\del^S\calQ_k^{\text{fib}}
$
is for $S\subset \underline k$.  By \refD{CartesianCube} this means that we want to know the connectivity of the map
\begin{equation}\label{E:HocolimCocartesianMap}
\underset{T\subsetneq S}{\hocolim}\Emb(Q, Q_T) \longrightarrow \Emb(Q, Q_S).
\end{equation}
To find the connectivity, we will use a general position argument.  This argument will only depend on dimensions of the manifolds and the size of $S$, so, without loss of generality, we may assume $S=\underline s=\{1,2,...,s\}$. 

Given a map 
$$
f\colon S^j\longrightarrow \Emb(Q, Q_S),
$$
we want to know whether it lifts to the punctured cube $T\mapsto\Emb(Q, Q_T)$, $T\subsetneq S$.  It if does, then it also lifts to the homotopy colimit above (see \eqref{E:SpaceToDiagramMapHocolim}). 
However, it suffices to consider the lift to $\Emb(Q, Q_\emptyset)$ since then composing with the maps in the punctured cube provides lifts to the rest of it.  In other words, in the diagram
$$
\xymatrix{
&   \Emb(Q,Q_\emptyset)\ar[d] \\
&   \Emb(Q,Q_T)\ar[d] \\
S^j\ar[r]^-f \ar[uur] \ar@{-->}[ur]& \Emb(Q,Q_S)
}
$$
if we have a lift $S^j\to \Emb(Q,Q_\emptyset)$ of $f$, then we also have the dotted map for all $T\subsetneq S$.  One way to think about this is that, if, given $s\in S^j$, we can produce an embedding of $Q$ that misses all the $Q_i$ (since this is what a point in $\Emb(Q,Q_\emptyset)$ is), then we have an embedding that misses any subset of the $Q_i$'s.

Now consider the adjoint of $f$,
$$
h\colon S^j\times Q \longrightarrow Q_S.
$$
Thus $h$ is a map that misses the images of $Q_i$ in $\R^n$, for $i\notin S$.  If we can arrange for it to also miss the images of $Q_i$ for $i\in S$, then this would mean that $h$ factors through $Q_\emptyset$, or, in other words, that $f$ factors through $\Emb(Q, Q_\emptyset)$.  

So consider the map
\begin{align*}
H\colon S^j\times \prod_{i\in S} Q & \longrightarrow \prod_{i\in S}Q_S \\
(v, x_1, ..., x_{s}) & \longmapsto (h(v,x_1), ..., h(v,x_{s})
\end{align*}  
By a small homotopy, we can arrange for $H$ to be transverse
to the submanifold
$$
\prod_{i\in S}e(Q_i) \subset \prod_{i\in S}Q_S.
$$
Since each $e(Q_i)$ is of dimension $l$ (as $Q_i$ is) and $Q_S$ is of dimension $n$, this submanifold is of codimension $sn-sl=s(n-l)$.  By transversality, the inverse image
$$
H^{-1}\left(\prod_{i\in S}e(Q_i)\right)
$$
is a submanifold of $S^j\times \prod_{i\in S} Q$ of the same codimension.  Its dimension is thus
$$
j+sl - s(n-l) = j+ s(2l-n).
$$
Therefore if $j<-s(2l-n)=s(n-2l)$, this submanifold is empty, which means that nothing maps to the product $\prod_{i\in S}e(Q_i)$, i.e.~$H$ misses all the necessary $e(Q_i)$.  This in turn means that, for such $j$, the map $h$ lands in $Q_\emptyset$, or, in other words, $f$ factors through $\Emb(Q,Q_\emptyset)$.

The argument could be repeated with homotopies
$$
S^j\times I \longrightarrow \Emb(Q,Q_S),
$$
but now the condition would be that $j<s(n-2l)-1$.  

%

Thus for $j<s(n-2l)$, the map $S^j \to \Emb(Q,Q_S)$ lifts to the punctured cube $T\mapsto \Emb(Q, Q_T)$, and hence to its homotopy colimit.  For $j<s(n-2l)-1$, the same happens for the map $S^j\times I \to \Emb(Q,Q_S)$.
Putting this together finally means that that 
for $j<s(n-2l)-1$, there is a bijection of homotopy classes 
$$
\pi_j(\underset{T\subsetneq S}{\hocolim}\Emb(Q, Q_T))\cong \pi_j(\Emb(Q,Q_S)),
$$
and for $j=s(n-2l)-1$, there is a surjection 
$$
\pi_{s(n-2l)-1}(\underset{T\subsetneq S}{\hocolim}\Emb(Q, Q_T))\twoheadrightarrow \pi_{s(n-2l)-1}(\Emb(Q,Q_S)).
$$
In other words, the map \eqref{E:HocolimCocartesianMap} is $s(n-2l)-1$-connected.  Or, 
$$
\del^S\calQ_k^{\text{fib}}\ \ \ \text{is}\ \ \  (s(n-2l)-1)\text{-cartesian}.
$$

Now we want to use \refT{GeneralB-M}, and we first observe that $(s(n-2l)-1)$ increases with $s$, satisfying the monotone increase condition for how cocartesian the face $\del^S\calQ_k^{\text{fib}}$ should be in terms of the size of $S$.  

We next look at the the sum of the connectivities of the faces over partitions of $k$.  Suppose $S_1$, ..., $S_a$ is a partition of $\underline k$ with sizes of the subsets $s_1$, ..., $s_a$.  Then the sum of connectivities of the corresponding faces is
$$
s_1(n-2l)-1+\cdots+s_a(n-2l)-1 = (s_1+\cdots+s_a)(n-2l)-a=k(n-2l)-a.
$$
This number is minimized for the largest $a$, which occurs when $a=k$, i.e.~when $\underline k$ is partitioned into singletons. \refT{GeneralB-M} thus finally says that the cube $\calQ_k^{\text{fib}}$, and hence also the cube $\calQ_k$, is
$$
(1-k+k(n-2l)-k = k(n-2l-2)+1)\text{-cartesian}
$$
as desired.
\end{proof}

\begin{proof}[Proof of \refT{TowerConvergenceToEmb2m}]
The first step in the proof is to upgrade the result of \refT{TowerConvergenceToEmbCartPQ} to the same statement but with handle dimensions replacing the dimensions.  The idea here is that if two manifolds are embedded in another one and the cores of the handles miss each other generically, then it can be arranged that the manifolds themselves miss each other by shrinking the handles.  The details for this kind of a reduction from dimension to handle dimension can be found in the discussion following Theorem A in \cite{GK}.

The argument now is an induction on handle dimension.  This is sensible since, as we saw in \refS{EmbeddingsTower}, $T_k\Emb(M,\R^n)$ can be defined inductively in terms of $T_k$'s of submanifolds of $M$ of lower handle dimension.  At the end of the iterative process, we have punctured cubical diagrams of embeddings of $\leq k$ balls in $\R^n$, and in that case, as noted in \eqref{E:SameOnBalls},
$$\Emb(\leq k \text{ balls}, \R^n)\stackrel{\simeq}{\longrightarrow} T_k\Emb(\leq k \text{ balls}, \R^n).
$$
At this stage, the holes $A_i$ are simply chosen to be the balls themselves, and so 
$$\Emb(\text{balls}, \R^n)\longrightarrow \underset{S\in \calP_0(\underline{k+1})}{\holim} T_k\Emb(\text{balls}-\bigcup_{i\in S}A_i, \R^n)\simeq\underset{S\in \calP_0(\underline{k+1})}{\holim} \Emb(\text{balls}-\bigcup_{i\in S}A_i, \R^n)
$$
is $(k(n-2)+1)$-connected (set $l=0$ in \refT{TowerConvergenceToEmbCartPQ}).  This is the base case for the induction.

For the inductive step, if $L$ is a submanifold of $M$ of some handle dimension $l$, look at the diagram
$$
\xymatrix{
\Emb(L, \R^n) \ar[r]\ar[d] &   T_k\Emb(L, \R^n) \ar[d]\\
\underset{S\in \calP_0(\underline{k+1})}{\holim} \Emb(L-\bigcup_{i\in S}A_i, \R^n) \ar[r] &
\underset{S\in \calP_0(\underline{k+1})}{\holim} T_k\Emb(L-\bigcup_{i\in S}A_i, \R^n)
}
$$
The right vertical map is an equivalence, as mentioned in \eqref{E:T_kSame}.  By induction, for each $S\neq \emptyset$, the map 
$$
\Emb(L-\bigcup_{i\in S}A_i, \R^n) \longrightarrow T_k\Emb(L-\bigcup_{i\in S}A_i, \R^n)
$$
is $(k(n-2(l-1)-2)+1)$-connected, since the handle index of $L-\bigcup_{i\in S}A_i$ is $l-1$.  Passing to homotopy limits over all nonempty $S$, it follows that the bottom horizontal map in the square is $(k(n-2(l-1)-2)+1-k)$-connected (see \cite[Proposition 1.22]{CalcII} or \cite[Proposition 5.4.17]{MV:Cubes}).

By \refT{TowerConvergenceToEmbCartPQ} (more precisely the version where the handle dimension replaces dimension), the left vertical map is $(k(n-2l-2)+1)$-connected.  It then finally follows by \refP{CompositionConnectivity} that the top horizontal map is also $(k(n-2l-2)+1)$-connected as desired.
\end{proof}

More details about the above proof can be found in \cite[Theorem 2.3]{GW:EI2} and \cite[Theorem 10.3.4]{MV:Cubes} (both versions of the proof also give a way of choosing the submanifolds $A_i$ so that the handle dimension is reduced when they are removed).


\subsection{Further comments and generalizations}\label{S:Further}


There are several generalizations that immediately follow.  Namely, the only ingredients in the proof of \refT{TowerConvergenceToEmbCartPQ} were the dimensions of the manifolds involved.  It was not necessary for $P$ or the $Q_i$ to be submanifolds of $L$ nor did they all need to be of same dimension.  Moreover, they could have been open or closed.  In addition, $Q=P\cup Q_{k+1}$ did not need to be a disjoint union of two manifolds.  Finally, the only information we used about $\R^n$ was its dimension, so this could have been any smooth manifold $N$ without boundary of dimension $n$; this in particular means that \refT{TowerConvergenceToEmb2m} is true with $N$ replacing $\R^n$.

Taking all this into account, the proof of \refT{TowerConvergenceToEmbCartPQ} can be repeated verbatim to yield the following more general statement.

\begin{thm}[Generalization of \refT{TowerConvergenceToEmbCartPQ}]\label{T:TowerConvergenceToEmbCartPQGeneral}
Suppose $Q_i$, $1\leq i\leq k+1$, are smooth manifolds of dimensions $q_i$ and $N$ is a smooth manifold of dimension $n$, all without boundary.  Then the $(k+1)$-cube 
\begin{align}
\calQ_k\colon \calP(\underline{k+1}) & \longrightarrow \Top \notag\\
S &\longmapsto \Emb(\bigcup_{i\notin S} Q_i, N) \notag
\end{align}
is 
$$
\text{$\left(\sum_{i=1}^{k}(n-q_i-q_{k+1}-2)+1\right)$-cartesian.}
$$
\end{thm}

This result appears as \cite[Proposition 9.7]{GK:EmbDisj}, although it is phrased there in terms of the equivalent statement about how cartesian the cube of fibers, $\calQ_k^{\text{fib}}$, is.  A short proof is given there, and another more complete proof appears as \cite[Theorem 10.3.11]{MV:Cubes}.  A generalization to manifolds with boundary is possible, although some care has to be taken with how the embeddings behave on the boundary.

However, an even stronger statement is true, namely that the cube $\calQ_k$ is 
$$
\left(\sum_{i=1}^k(n-q_i-2)-q_{k+1}+1)\right)\text{-cartesian},
$$
where $q_i$ is not just the dimension of $Q_i$, but its handle dimension (going from dimension to handle dimension is not hard, and it was already mentioned in the proof of \refT{TowerConvergenceToEmb2m}).
This result is due to Goodwillie and Klein \cite[Theorem A]{GK}.  The proof uses a similar result for Poincar\'e embeddings \cite{GK:EmbDisj} and, in the authors' own words, requires ``ideas from homotopy theory, surgery theory, and pseudoisotopy theory''; the many ingredients make for a complex proof that is beyond the intended level of difficulty of this paper.  Goodwillie and Klein also take a general $N$ rather than $\R^n$ and allow for all manifolds to have boundary.  



As Goodwillie and Weiss then show in \cite[Corollary 2.5]{GW:EI2}, the Goodwillie-Klein result leads to the following stronger version of \refT{TowerConvergenceToEmb2m}.

\begin{thm}[Convergence of the Taylor tower to embeddings, strong version]\label{T:TowerConvergenceToEmb}
Suppose $M$ is a smooth closed manifold of dimension $m$.
 Then, for $k\geq 1$, the map 
$$
\Emb(M,\R^n)\longrightarrow T_{k}\Emb(M,\R^n)
$$
is
$$
\text{$(k(n-m-2)-m+1)$-connected.}
$$
Thus if $n>m+2$, the Taylor tower for $\Emb(M,\R^n)$ converges to this space, i.e.~the map $\Emb(M,\R^n)\to T_{\infty}\Emb(M,\R^n)$ is an equivalence.
\end{thm}

Because the Goodwillie-Klein result is more general, \refT{TowerConvergenceToEmb} is also true more generally for any $N$ in place of $\R^n$, and all the manifolds can have boundary (the case of manifolds with boundary is treated separately in \cite[Section 5]{GW:EI2}).  Furthermore, one can replace $M$ by the interior of any submanifold of codimension zero, in which case $m$ is replaced by its handle dimension (this includes $M$ itself); this is in fact already apparent from our proof of \refT{TowerConvergenceToEmb2m} since that part of the proof does not depend on particular connectivity numbers.


An overview of the arguments that lead to \refT{TowerConvergenceToEmb} can be found in \cite[Section 3]{GKW:EmbDisjSurg} as well as in \cite[Section 10.3]{MV:Cubes} (although the Goodwillie-Klein result is beyond the scope of the latter work as well).

The proof of the above strong convergence result does not utilize \refT{TowerConvergence}, as our weaker statement,
\refT{TowerConvergenceToEmbWeak}, does.  Namely, the connectivity numbers are proved directly and the convergence follows under the assumption $n>m+2$. With this in hand, one can then in fact deduce \refT{TowerConvergence} as a consequence of \refT{TowerConvergenceToEmb}.
%
%
%
%
%
More precisely, given \refT{TowerConvergenceToEmb}, we have a diagram
$$
\xymatrix{
\Emb(M,\R^n)\ar[rrr]^-{k(n-m-2)-m+1} \ar[drrr]^(0.5){\ \ \ \ \ \ \ \ (k-1)(n-m-2)-m+1}& & & T_k \Emb(M,\R^n) \ar[d]\\
& & & T_{k-1}\Emb(M,\R^n)
}
$$
with the connectivities of the maps indicated.  But then by \refP{CompositionConnectivity}(2), it follows that $T_k\Emb(M,\R^n)\to T_{k-1}\Emb(M,\R^n)$ must have connectivity that is the lower of the two numbers, which is $(k-1)(n-m-2)-m+1$ and this is precisely what \refT{TowerConvergence} asserts.

What we thus have is that \refT{TowerConvergence} is subsumed by \refT{TowerConvergenceToEmb}, but, in our setup, the former still has merit since it has allowed us to improve the connectivites of the maps to the tower from what \refT{TowerConvergenceToEmb2m}, the weak convergence result, initially gave us --- we were able to bring them up to $k(n-m-2)-m+1$ from $k(n-2m-2)+1$ (using \refP{ConnectivityAssumeConvergence}).  What we are not able to do is improve the assumption $n>2m+2$ to the better one $n>m+2$; this is where the proof of \refT{TowerConvergenceToEmb} becomes too demanding for the scope of this paper.

\bibliographystyle{alpha}

\bibliography{/Users/ismar/Desktop/Math/Bibliography}

\end{document}